\documentclass{article}
\usepackage[a4paper, margin=1in]{geometry}

\usepackage{tikz-cd}
\usepackage{amsthm}
\usepackage{amssymb}
\usepackage{amsmath}
\usepackage{mathrsfs}
\usepackage{xcolor}
\usepackage[euler-digits, euler-hat-accent]{eulervm}
\usepackage[backend=biber, style=nature, sorting=nyt]{biblatex}
\usepackage[hidelinks]{hyperref}
\usepackage{cleveref}
\newtheorem{thm}{Theorem}[section]
\newtheorem{lemma}[thm]{Lemma}
\newtheorem{prop}[thm]{Proposition}
\newtheorem{cor}[thm]{Corollary}

\theoremstyle{definition}

\newtheorem{rmk}{Remark}[thm]

\newcommand{\cg}[1]{\mathscr{#1}}

\newcommand{\mb}[1]{\mathbold{#1}}

\newcommand{\Z}{\mathbb{Z}}

\newcommand{\proj}{\mathbb{P}}
\newcommand{\aff}{\mathbb{A}}
\newcommand{\oh}{\mathscr{O}}

\newcommand{\Spec}{\operatorname{Spec}}
\newcommand{\Proj}{\operatorname{Proj}}
\newcommand{\Sym}{\operatorname{Sym}}
\newcommand{\Hom}{\operatorname{Hom}}

\newcommand{\End}{\operatorname{End}}

\newcommand{\id}{\operatorname{id}}
\newcommand{\im}{\operatorname{im}}
\newcommand{\M}{\mb{M}}
\newcommand{\rd}{\mathrm{red}}
\newcommand{\tgt}{\operatorname{Tan}}

\newcommand{\epi}{\twoheadrightarrow}
\newcommand{\mono}{\hookrightarrow}
\newcommand{\nra}[1]{\xrightarrow{#1}}

\title{Motives of Certain Hyperplane Sections of Milnor Hypersurfaces}
\author{Evan Marth\footnote{supported by NSERC Discovery grant RGPIN-2022-03060 and an Ontario Graduate Scholarship} }

\begin{document}

\maketitle

\begin{abstract}
  We construct a hyperplane section $Y$ of a Milnor hypersurface associated to a regular semisimple endomorphism $\varphi$. Exploiting its structure as a hyperplane section 
  of a projective bundle and its natural torus action, we give a motivic decomposition of $Y$, which encodes both the 
  cellular structure of $Y$ and the arithmetic of the eigenvalues of $\varphi$. This decomposition is proven without using the ``nilpotence principle", that is to 
  say there are no ``phantoms".
\end{abstract}

\section{Introduction} 

Let $V$ be a vector space of dimension $n+1$ over a field $k$. Consider the partial flag variety \[ E = \{ W_1 \subseteq W_n \subseteq V : \dim W_i = i \} \] 
Let $\varphi \in \End(V)$ be an endomorphism with distinct eigenvectors, so in particular the subalgebra $L := k[\varphi] \subseteq \End(V)$ is étale of dimension $n+1$ over $k$.
Now define the hyperplane section of $E$ \[ Y = \{ W_1 \subseteq W_n \subseteq V : \dim W_i = i, \quad \varphi(W_1) \subseteq W_n \} \]

The main result of this paper is 

\begin{thm} \label{main_thm}
  If $L$ is a finite product $\prod_i K$, for a fixed Galois extension $K/k$, and $n \geq 1$, then the motive of $Y$ decomposes as 
  \[ \M(Y) = \bigoplus_{i=0}^{n-2} \M(\proj^n_k)(i) \oplus \M(\Spec L)(n-1) \] 
\end{thm}

This gives a proof of a special case of the main theorem of \cite{twistedmilnor} without ``phantoms" -- motives in the decomposition which become trivial after 
a field extension. 
In fact, this was the main inspiration for the present article and many of the results of \cite{twistedmilnor} appear here, if only implicitly. 

Our approach first considers the problem in the general setting of an $\oh(1)$-type divisor of a projective bundle in Section 2, where we obtain a criterion (\Cref{crit}) for 
a decomposition as in \Cref{main_thm} to hold. The criterion is then verified for $Y$ in Section 3 using equivariant methods for torus actions. As a consequence, we obtain a
natural proof showing that if $\varphi$ satisfies the hypotheses of \Cref{main_thm}, $L$ (or equivalently $K$) is an invariant of $Y$ (\Cref{invariant}).

\vspace{1em}
\noindent \textbf{Notation and conventions:} A smooth variety over a field $k$ is an equidimensional algebraic scheme which is smooth over $k$. 
For a vector space $V$ over $k$ 
and a $k$-scheme $s: X \to \Spec k$, $\underline{V}$ is the trivial vector bundle $s^*(\tilde{V})$ on $X$. $\proj(V)$ is the projective space of one-dimensional quotients 
of $V$, and similarly $\proj(\cg{E})$ is the projective bundle of rank one quotients of the vector bundle $\cg{E}$. A Cartier divisor linearly equivalent to 
the zero locus of a section of a line bundle $\cg{L}$ is said to be of $\cg{L}$-type. 
The $i$-th power of the Tate motive is written $\Z(i)$, and ``twists" in the opposite direction of the Tate twist of $\ell$-adic cohomology.

\section[Chow groups of O(1)-type diviors on a projective bundle]{Chow groups of $\oh(1)$-type divisors on a projective bundle}

Let $k$ be an arbitrary field. We consider the following situation: $X$ is a smooth projective variety over $k$ with a vector bundle $\cg{E}$ of rank $r+1$ which is
generated by global sections. Set $E = \proj(\cg{E}) = \Proj (\operatorname{Sym} \cg{E})$ with projection map $\pi$. We will be concerned with sections $s \in H^0(X, \cg{E}) = H^0(E, \oh(1))$ such that the zero locus $Z$ of  
$s$ is smooth of codimension $r+1$ in $X$ and $Y$, the divisor corresponding to $s$, is smooth. Letting $U = X - Z$,
$Y|_U := Y \times_X U$ is a projective bundle of rank $r-1$ over $U$ (corresponding to the cokernel $\cg{F}$ of $\oh_X \nra{s} \cg{E}$, restricted to $U$) 
and $Y|_Z = E|_Z$ is a projective bundle of rank $r$ over $Z$ (corresponding to $\cg{E} \otimes \oh_Z$). The following commutative diagram summarises our notation 
for the inclusion maps: 
\[\begin{tikzcd}[cramped]
	{E|_Z} && E && {E|_U} \\
	\\
	&& Y && {Y|_U}
  \arrow["{j'}"', hook, from=1-1, to=1-3]
	\arrow["{j}", hook, from=1-1, to=3-3]
	\arrow[hook', from=1-5, to=1-3]
	\arrow["i", hook', from=3-3, to=1-3]
	\arrow["{i'}"', hook', from=3-5, to=1-3]
	\arrow[hook', from=3-5, to=1-5]
	\arrow[hook', from=3-5, to=3-3]
\end{tikzcd}\]

By the projective bundle theorem \cite[Theorem 9.6]{3264}, we have 
that $A^\bullet(E)$ is a free $A^\bullet(X)$-module generated by $H_E^i$, $H_E = c_1(\oh(1))$ and $i=0, \dots, r$.  
Analoguous statements hold for $A^\bullet(Y|_U)$, $A^\bullet(U)$ and $H_Y = i'^* H_E$  
(rank $r$), and $A^\bullet(E|_Z)$, $A^\bullet(Z)$ and $H = j'^* H_E$ (rank $r+1$). The former requires some explanation. By construction, $Y|_U$, as a $U$-scheme, is 
$\proj(\cg{F}|_U)$. The inclusion into $E|_U = \proj(\cg{E}|_U)$ comes from the surjective homomorphism $\Sym \cg{E}|_U \epi \Sym \cg{F}|_U$ induced 
by the quotient map $\cg{E} \epi \cg{F}$.
Hence the line bundle $\oh_{\proj(\cg{F}|_U)}(1)$ on $Y|_U$ is the pullback of $\oh_{\proj(\cg{E}|_U)}(1)$ by this inclusion 
(this follows from the local case as in \cite[Proposition II.5.13.c]{hartshorne}). 
Since $\oh_{\proj(\cg{E}|_U)}(1) = \oh_{\proj(\cg{E})}(1)|_U$,  
the identity $H_Y = i'^* H_E$ is the translation of the above for the first Chern classes of the line bundles.

Given a smooth projective variety $S$, we can take the data $X, \cg{E}, s$ and associate to it the data $X \times_k S, p_1^* \cg{E}, p_1^* (s)$. 
Then, applying the above constructions to  $X \times_k S, p_1^* \cg{E}, p_1^* (s)$, we see that the varities $E, Z, Y$
are obtained from those constructed from $X, \cg{E}, s$ by taking a product with $S$. The same holds for 
morphisms and classes in the Chow rings. All of these operations will simply be called ``base change by $S$".

\begin{rmk}
  It is harmless to take $Z_{\rd}$ in this setup instead of $Z$. Indeed, $U$ remains the same, and all other data are unaffected. The description of the ``base change"
  of $Z$ will still work, since a smooth variety $S$ is geometrically reduced, hence $Z_{\rd} \times_k S = (Z \times_k S)_{\rd}$. We will denote both by $Z$ is the sequel.
\end{rmk}

Define a group homomorphism $\varphi: A^\bullet(E) \oplus A^\bullet(E|_Z) \to A^\bullet(Y)$ by
 $\varphi = (i^*, j_*)$ 

\begin{prop} \label{hyp_chow}
  With the same notations as above:
  
  \textbf{a.} $\varphi$ is surjective

  \textbf{b.} For every class $\gamma \in A^\bullet(Y)$, there exist $\alpha_0, \dots, \alpha_{r-1} \in A^\bullet(X)$ and $\beta \in A^\bullet(Z)$ such that 
  \[ \varphi(\sum_{i=0}^{r-1} \pi^* \alpha_i \cdot H_E^i, \pi|_Z^* \beta) = \gamma \]

  \textbf{c.} If $j_* \circ (\pi|_Z)^*$ is injective, then such an element is unique.
\end{prop}

\begin{proof}
  \textbf{a.} By the right exact sequence $A^\bullet(E|_Z) \nra{j_*} A^\bullet(Y) \to A^\bullet(Y|_U) \to 0$, we are reduced to showing that $i'^*$ is surjective. 
  $A^\bullet(Y|_U)$ is generated (as a ring) by $(\pi|_U)^* A^\bullet(U)$ and $H_Y$. Clearly $H_Y$ is in the image of $i'^*$ and the commutativity of 
\[\begin{tikzcd}[cramped]
	{A^\bullet(E)} && {A^\bullet(Y|_U)} \\
	\\
	{A^\bullet(X)} && {A^\bullet(U)}
	\arrow["{i'^*}", from=1-1, to=1-3]
	\arrow["{\pi^*}", from=3-1, to=1-1]
	\arrow[from=3-1, to=3-3]
	\arrow["{\pi|_U^*}"', from=3-3, to=1-3]
\end{tikzcd}\] and the surjectivity of the restriction $A^\bullet(X) \to A^\bullet(U)$ shows that  $(\pi|_U)^* A^\bullet(U)$ is also in the image.

\textbf{b.}   That we can eliminate positive powers of $H$ follows from the equality $j_*(H \cdot \alpha) = i^*(j'_*(\alpha))$. 
  This identity holds since $i^* \circ j'_* = (i^* \circ i_*) \circ j_*$ and 
  since $Y$ is a divisor we have $(i^* \circ i_*)(\beta) = i^*(H_E) \cdot \beta$ (\cite[Proposition 2.6.c]{fulton}). But by the projection formula, 
  $i^* H_E \cdot j_*(\alpha) = j_* (j^*(i^* (H_E)) \cdot \alpha ) = j_*(H \cdot \alpha)$. 

  Now let $\alpha \in A^\bullet(X)$, then $i'^*(\pi^*(\alpha) \cdot H_E^r) = \sum_{i=0}^{r-1} (\pi|_U)^* \gamma_i \cdot H_Y^i $ by the projective bundle theorem, 
  hence by the proof of \textbf{a.} there 
  are elements $\alpha_0, \dots, \alpha_{r-1}$ such that $i'^*(\sum_{i=0}^{r-1} \pi^* \alpha_i \cdot H_E^i ) = \sum_{i=0}^{r-1} (\pi|_U)^* \gamma_i \cdot H_Y^i$. 
  Hence $i'^*(\pi^* \alpha \cdot H_E^r - \sum_{i=0}^{r-1} \pi^* \alpha_i \cdot H_E^i) = 0$, so 
  $i^*(\pi^* \alpha \cdot H_E^r) - i^*(\sum_{i=0}^{r-1} \pi^* \alpha_i \cdot H_E^i) \in \im j_*$. So there are $\beta_0, \dots, \beta_r$ such that 
  $\varphi(\sum_{i=0}^{r-1} \pi^* \alpha_i \cdot H_E^i, \sum_{j=0}^{r} (\pi|_Z)^*\beta_j \cdot H^i) = i^*(\pi^* \alpha \cdot H_E^r)$. Eliminating the positive powers of $H$ as above 
  will then give an element of the desired form.

  \textbf{c.} Elements of the form given in \textbf{b.} form a subgroup in $A^\bullet(E) \oplus A^\bullet(E|_Z)$, so we just need to prove $\ker \varphi$ meets this 
  subgroup trivially. If $\varphi(x) = 0$, then $i_*(\varphi(x)) = 0$. We have $i_*(i^*(\sum_{i=0}^{r-1} \pi^* \alpha_i \cdot H_E^i)) = \sum_{i=0}^{r-1} \pi^* \alpha_i \cdot H_E^{i+1}$
  . Let $\hat{j}: Z \mono X$ denote the inclusion of $Z$ in $X$. We have $i_*(j_*((\pi|_Z)^* \beta )) = j'_*((\pi|_Z^* \beta)) = \pi^*(\hat{j}_* \beta)$ since $\pi$ is flat and 
\[\begin{tikzcd}[cramped]
	{E|_Z} && E \\
	\\
	Z && X
	\arrow["{j'}", hook, from=1-1, to=1-3]
	\arrow["{\pi|_Z}", from=1-1, to=3-1]
	\arrow["\pi"', from=1-3, to=3-3]
	\arrow["{\hat{j}}", hook, from=3-1, to=3-3]
\end{tikzcd}\] is a fibre square (by definition!). Putting these two facts together, we find that 
\[ \left(\sum_{i=0}^{r-1} \pi^* \alpha_i \cdot H_E^i, (\pi|_Z)^* \beta \right) \in \ker \varphi \implies  \pi^*(\hat{j}_* \beta) + \sum_{i=0}^{r-1} \pi^* \alpha_i \cdot H_E^{i+1} = 0 \]
 Since $1, H_E, \dots, H_E^r$ are a $A^\bullet(X)$ basis, this implies $\alpha_i = 0$ for $i=0, \dots, r-1$. So we must have $j_*(\pi|_Z^* \beta) = 0$, hence by hypothesis 
 $\beta = 0$, so the intersection with the kernel is trivial as desired.
\end{proof}

We will call this subgroup $C$, and note that \textbf{c.} just says that $C$ is mapped isomorphically onto $A^\bullet(Y)$ by $\varphi$. Turning to the motive of $Y$, we 
use the category defined in \cite{manin}. In the case of the Chow ring, known in the literature as the \emph{category of effective Chow motives}. We also follow the notation 
of \cite[\S 3]{manin} concerning the ``Identity principle". 

Now consider the full subcategory of the category of effective Chow motives whose objects are finite direct sums of motives 
$\M(X)(i) = \M(X) \otimes \Z(i)$, with $X$ a smooth projective variety .
By means of direct sums, all hom-sets in this categoy can be obtained from correspondences in the graded groups 
$A^\bullet(X \times_k Y)$ for $X, Y$ smooth projective varieties. So we get a variant of Yoneda's lemma for this subcategory, where we only need consider these correspondences. 
More precisely, we will use the following consequence: if $M, N$ are such motives, and $\psi \in \Hom(M, N)$, then if for all smooth projective varieties $S$ over $k$,
$\psi_S : \Hom^\bullet(\M(S), M) \to \Hom^\bullet(\M(S), N)$ is an isomorphism, then $M \cong N$.

Let $x = i^*(H_E) \in A^1(Y)$, $f = \pi \circ i$ and $g = \pi|_Z$. Then we have correspondences: 
\begin{align*} &c_{x} \in \Hom^1(\M(Y), \M(Y)), \qquad c_f \in \Hom(\M(X), \M(Y)), \\ 
&c_g \in \Hom(\M(Z), \M(E|_Z)), \qquad c_j^t \in \Hom^r(\M(E|_Z), \M(Y)) 
  \end{align*} and we define for $0 \leq i \leq r-1$ correspondences \[ f_i = c_{x}^{(i)} \circ c_f \in \Hom(\M(X)(i), \M(Y)), \qquad f' = c_j^t \circ c_g \in \Hom(\M(Z)(r), \M(Y)) \]
  and a morphism \[ \psi: M := \bigoplus_{i=0}^{r-1} \M(X)(i) \oplus \M(Z)(r) \to \M(Y), \qquad \psi = (f_0, \dots, f_{r-1}, f')  \]

  For any smooth projective variety $S$, we can factor $\psi_S$ through the $\varphi$ of \Cref{hyp_chow} applied to varieties, bundle, section, etc ``base changed" by $S$, such that 
  $\Hom^\bullet(\M(S), M)$ maps isomorphically onto the distinguished subgroup $C \subseteq A^\bullet(E \times_k S) \oplus A^\bullet(E|_Z \times_k S)$. Using \Cref{hyp_chow} \textbf{c.}
  we obtain 

  \begin{cor} \label{crit}
    If  $(j \times \id_S)_* \circ (\pi|_Z \times \id_S)^*$ is injective for all smooth projective varieties $S$, then $\psi$ is an isomorphism.
  \end{cor}

\section{Application to the hyperplane section}

Let $k$ be a field and $V$ a vector space over $k$ of dimension $n+1$, $n \geq 1$. We being by describing more precisely the varieties from the introduction: 

The natural pairing $V \times V^* \to k$ defines a hypersurface $E \subseteq \proj(V^*) \times_k \proj(V)$, 
called a Milnor hypersurface. Concretely, fixing a basis $y_0, \dots, y_n$ of $V$ and corresponding dual basis $x_0, \dots, x_n$ of $V^*$, $E$ is defined by the equation 
$\sum_{i=0}^n x_i y_i = 0$, or what is the same, $E$ is the divisor given by the section $\sum_{i=0}^n x_i \otimes y_i \in 
H^0(\proj(V) \times_k \proj(V^*), p_1^*\oh_{\proj(V^*)}(1) \otimes p_2^* \oh_{\proj(V)}(1))$. 
Restricting the first projection map $p_1$ to $E \mono \proj(V^*) \times_k \proj(V)$, we obtain a 
projective bundle $\pi: E \to \proj(V^*)$ (the restriction of $p_2$ will be denoted $\pi'$). 
Indeed, one sees $E \cong \proj(\underline{V} / \oh_{\proj(V^*)}(-1))$, with $\oh(1) = \pi'^* \oh_{\proj(V)}(1)$, via the inclusion 
$E \mono \proj(V^*) \times_k \proj(V)$. Consequently, with $\cg{E} = (\underline{V} / \oh_{\proj(V^*)}(-1)) \otimes \oh_{\proj(V^*)}(1)$, 
we also have $E \cong \proj(\cg{E}$) over $\proj(V^*)$ and 
$\oh_{\proj (\cg{E})}(1) = \pi^* \oh(1)_{\proj(V^*)} \otimes \pi'^* \oh(1)_{\proj(V)}$ (see \cite[Lemma II.7.9]{hartshorne}). 

$Y$ is then defined by additionally imposing the equation coming from the twisted pairing $V \times V^* \to k$, $(v, f) \mapsto f(\varphi(v))$, so in particular, intersecting 
with a divisor of type $\oh(1, 1) = p_1^*\oh_{\proj(V^*)}(1) \otimes p_2^* \oh_{\proj(V)}(1)$.

\begin{lemma}
  $Y$ is a smooth effective divisor in $E$ corresponding to a global section $s \in H^0(E, \oh_{\proj (\cg{E})}(1))$.
\end{lemma}

Note that this implies $Y$ is very ample, hence the terminology ``hyperplane section".
\begin{proof}
  By the construction of $Y$, if it is a divisor on $E$, then it is of $\oh_{\proj(\cg{E})}(1)$-type. Let $\alpha_0, \dots, \alpha_n$ be the 
  $n+1$ distinct eigenvectors of $\varphi$. We may assume $k$ algebraically closed, so in particular we have $\alpha_i \in k$. 
  Choosing a basis $y_0, \dots, y_n \in V$ diagonalizing $\varphi$, and letting $x_0, \dots, x_n \in V^*$ be the dual basis of the $y_i$, 
  we see that $Y$ is cut out by the polynomials \[ \sum_{i=0}^n x_i y_i = 0, \qquad \sum_{i=0}^n \alpha_i x_i y_i = 0\] in $P = \proj(V^*) \times_k \proj(V)$. 
  
  We proceed by applying the Jacobian criterion to
  $Y \subseteq P$ away from $x_i y_j = 0$. For $0 \leq i \neq j \leq n$, let $U_{i j} \subseteq P$ denote the open set of points where $x_i y_j \neq 0$.
  It is affine since $\oh(1,1)$ is very ample. Thus $U_{i j} \cong \aff^{2n-2}$, with coordinate ring
  $k[x'_0, \hdots, x'_n, y'_0, \hdots, y'_n]$, $x'_l = \frac{x_l}{x_i}, y'_l = \frac{y_l}{y_j}$. The $U_{i j}$ cover $Y$ since if there is only one $i$ with 
  $x_i$ or $y_i$ non-zero at a point of $Y$, then $\sum_{l=0}^n x_l y_l \neq 0$.
  $Y \cap U_{ij}$ is given by the equations $ \sum_{l=0}^n x'_l y'_l$ and $ \sum_{l=0}^n \alpha_i x'_l y'_l$ so the matrix
  $$ \begin{pmatrix}
    y'_j & x'_i \\ \alpha_j y'_j & \alpha_i x'_i
  \end{pmatrix} $$ appears as a $2 \times 2$ submatrix of the Jacobian matrix, (with $y'_j \neq 0$ and $x'_i \neq 0$ by definition). 
  Distinctness of the $\alpha_l$ shows that this matrix
  is non-singular. Thus the rank of the Jacobian matrix is 2, so $Y$ is a smooth divisor on $E$. 
\end{proof}

\begin{lemma} \label{compute_Z}
  The reduced zero locus $Z$ of the $s$ of the previous lemma in $\proj(V^*)$ is isomorphic to $\Spec L$.
\end{lemma}

\begin{proof} 
  Over an algebraic closure $\bar{k}/k$, the closed points of the zero locus of $s$ are those which have a fibre of dimension $n-1$. 
  With homogeneous coordinates $x_i, y_j$ as before, these are the points $[c_0, \dots, c_n]$ such that $\sum_{i=0}^n c_i y_i$ and $\sum_{i=0}^n \alpha_i c_i y_i$ 
  are linearly dependent. Since the $\alpha_i$ are distinct, this happens precisely when $c_i = 0$ for all but one value of $i$.
  In coordinate-free terms, these correspond to the eigenspaces of $\varphi \otimes 1$ in $V \otimes_k \bar{k}$, which are already defined in $\proj(V^*)(K)$. 
  Since $K$ is separable over $k$ and we assume $Z$ reduced, $Z \times_k K$ is reduced, hence $Z \times_k K \cong \coprod_{0 \leq i \leq n} \Spec K$ by the above description 
  of $Z(\bar{k})$.
  This implies in turn that $Z \cong \Spec B$, where $B$ is a reduced finite algebra over $k$ with $\dim_k B = n+1$ and 
  that $B$ is a direct product of field extensions of $k$,  
  $\prod_{1 \leq i \leq m} K_i$, where all $K_i$ have an inclusion into $K$. By definition of $K$, each eigenvalue of $\varphi$ has $\dim_k K$ conjugates. Since eigenspaces 
  of conjugate eigenvalues are conjugate, this implies the orbit of any point of $Z(K)$ under the action of $\mathrm{Gal}(K/k)$ has cardinality $\dim_k K$. 
  Hence, $\dim_k K_i = \dim_k K$ for all $1 \leq i \leq m$, so $K_i \cong K$. Now since $\dim_k B = \dim_k L = n+1$, we see that $B \cong L$, as desired.
\end{proof}

If the eigenvalues of $\varphi$ are in $k$, then $V$ decomposes into one-dimensional eigenspaces $V_i$, $0 \leq i \leq n$. This gives a torus
$T \subseteq \mathrm{GL}(V)$ consisting of the elements which send the $V_i$ into themselves, which acts on $\proj(V^*)$ and
$\proj(V)$ via the trivial and dual representations, respectively. 
These actions are then such that $E$ is $T$-stable under the induced $T$-action on $\proj(V^*) \times_k \proj(V)$, and so is 
$Y$ since $\varphi \otimes 1$ commutes with the elements of $T(\bar{k})$.
From \Cref{compute_Z}, we see that $E|_Z$ consists of $n+1$ copies of $\proj^{n-1}_k$, each of which is $T$-stable. These are precisely the fibres 
$E_i = \pi^{-1}([V_i])$, $0 \leq i \leq n \subseteq Y$. 

\begin{prop} \label{ortho}
  For $0 \leq i, j \leq n$, let $\gamma_i = [E_i] \in A^{n-1}(Y)$. Then $\deg(\gamma_i \cdot \gamma_j) =  \delta_{i j} (-1)^{n-1}$.  
\end{prop}

Since the degree of a class in the zeroth Chow group of a proper variety is invariant under change of base field, we may assume that $k$ is algebraically closed. This allows 
for the use of \emph{localisation} (\cite[Corollary 2.3.2]{teq}) for $T$-equivariant Chow groups to prove the proposition. 
To this end, we first gather some facts about the $T$-action on $Y$. 

For each $0 \leq i \leq n$, 
we have a homomorphism $t_i : T \to \mathrm{GL}(V_i) = \mathbb{G}_m$. The $t_i$ generate the character group $M$ of $T$, and we write $\chi_{i j}$ for $t_j - t_i \in M$.
We let $R = \Sym_\Z M = \Z[t_0, \dots, t_n]$ 
and $Q$ be the field of fractions of $R$. For $0 \leq i \neq j \leq n$, let $z_{i j} = ([V_i], [\bigoplus_{0 \leq l \neq j \leq n} V_l]) \in E$. These are also contained in $Y$ 
and $E|_Z$ and are the $T$-fixed points of these varieties.

\begin{lemma} \label{chars}
  The weights of the $T$-module $\tgt_{z_{i j}}(Y)$ are $\chi_{l j}$ and $\chi_{i l}$ for $0 \leq l \leq n$, $l \neq i, j$.
  and the submodule $\tgt_{z_{i j}} (E|_Z) \subseteq \tgt_{z_{i j}}(Y)$ is spanned by the weight spaces of $\chi_{l j}$, $0 \leq l \leq n, l \neq i, j$.
\end{lemma}

\begin{proof}
  For fixed $i \neq j$, for any $l \neq i, j$, the codimension $2$ subspace $\bigoplus_{0 \leq s \leq n, s \neq j, l} V_s \subseteq V$ corresponds to a $T$-stable line $L_{l j} 
  \subseteq \proj(V)$. Clearly $C_{l j} = \{ [V_i] \} \times L_{l j} \subseteq Y$ is $T$-stable and it is an easy computation that $T$ acts on 
  $\tgt_{z_{ij}}(C_{l j})$ by $\chi_{l j}$. Similarly, 
  one defines a line $L_{i l} \subseteq \proj(V^*)$ corresponding to $V_i \oplus V_l$, and sets $C_{i l} = L_{i l} \times \{ [\bigoplus_{0 \leq s \neq j \leq n} V_s] \} 
  \subseteq Y$.\footnote{The $T$-stable curves used in the proof are given in \cite[\S 3.1]{bp} in the case where $E$ is any adjoint variety.}
  Once again, it is easily verified that $T$ acts on $\tgt_{z_{i j}}(C_{i l})$ by $\chi_{i l}$. These are all weights of $\tgt_{z_{i j}}(Y)$ by the canonical 
  inclusions of the tangent spaces of the $T$-stable curves, and they make up all weights since
  $\dim_k \tgt_{z_{ij}}(Y) = 2n -2$. The characterisation of $\tgt_{z_{i j}}(E|_Z)$ follows since
  $\dim_k  \tgt_{z_{i j}}(E|_Z) = n-1$ and each of the $C_{l j}$ is contained in $E|_Z$.  
\end{proof}

We denote the $T$-equivariant Chow ring of a smooth $T$-variety $X$ by $A^\bullet_T(X)$ (for a general reference on equivariant intersection theory, see \cite{eq_int}). 
We write $\bar{\alpha}$ for the image of an element $\alpha$ under the forgetful map $A_T^\bullet(X) \to A^\bullet(X)$.
If $X$ is proper over $k$ with structure morphism $p$, 
we have the equivariant Poincaré pairing $\langle \cdot , \cdot \rangle_T: A_T^\bullet(X) \times A_T^\bullet(X) \to A_T^\bullet(\Spec k) = R$ defined by 
$(\alpha, \beta) \mapsto p_*(\alpha \cdot \beta)$. Notice that 
$\overline{\langle \alpha, \beta \rangle}_T = \deg(\bar{\alpha} \cdot \bar{\beta})$ by the naturality of the forgetful map (we extend $\deg$ to all of $A^\bullet(X)$ by 
setting it to $0$ for cycles of dimension greater than $0$. This extension is of course just $p_*$). 

Let $x \in X$ be a $T$-fixed point such the weights $\chi_1, \dots, \chi_m$ of $\tgt_x(X)$ are non-zero. 
Following \cite[Theorem 4.2]{teq}, we define the \emph{equivariant multiplicity}
of a cycle $\alpha \in A_T^\bullet(X)$ $e_{x, X}(\alpha)$ to be the image of $\alpha$ by the unique $R$-linear map $e_{x, X} : A_T^\bullet(X) \to Q$ such that 
$e_{x, X}([x]) = 1$ and $e_{x, X}([X']) = 0$ for any $T$-invariant subvariety $X' \subseteq X$ which does not contain $x$. The smoothness of $X$ implies that 
$e_{x, X}([X]) = ( \prod_{1 \leq i \leq m} \chi_i )^{-1}$.
Moreover, for smooth $X'$, 
$e_{x, X}([X']) = e_{x, X'}([X'])$. 

\begin{lemma} \label{eq_ids}
  For $\alpha \in A_T^\bullet(Y)$ and $0 \leq i \neq j \leq n$, let $\alpha_{i j}$ be the pullback of $\alpha$ by the inclusion $\{ z_{ij} \} \mono Y$.  
  We have the following identities:
  \begin{equation} e_{z_{i j}, Y}(\alpha) = \frac{ \alpha_{i j} } {\prod_{l \neq i, j} \chi_{i l} \chi_{l j}}  \end{equation}
  \begin{equation} \langle \alpha, \beta \rangle_T = \sum_{0 \leq i \neq j \leq n} \frac{\alpha_{i j} \beta_{i j}}{\prod_{l \neq i, j} \chi_{i l} \chi_{l j} } \end{equation} 
\end{lemma}

\begin{proof}
  For (1), let $\iota_{i j}: \{ z_{i j} \} \mono Y$, $\iota: Y^T \mono Y$ be the obvious inclusion maps. By \cite[Corollary 4.2]{teq} and \Cref{chars}, 
  \[ [Y] = \sum_{0 \leq i \neq j \leq n} \frac{1}{\prod_{l \neq i, j} \chi_{i l} \chi_{l j} } [z_{i j}], \qquad 
    \alpha = \sum_{0 \leq i \neq j \leq n} e_{z_{i j}, Y}(\alpha) [z_{ij}]  \] in $A_T^\bullet(Y) \otimes_R Q$.
    Using the identification $A_T^\bullet(Y^T) \otimes_R Q = \bigoplus_{0 \leq i \neq j \leq n} Q $ coming from 
    the inclusion of each fixed point into $Y^T$, we can rewrite these equalities as
    \[ [Y] = \iota_*  \left( \frac{1}{\prod_{l \neq i, j} \chi_{i l} \chi_{l j} } \right)_{i j}, \qquad 
    \alpha = \iota_*  \left( e_{z_{ij}, Y}(\alpha) \right)_{i j}  \] But $\alpha = \alpha \cdot [Y]$, so by the projection formula we have 
    \[ \iota_*  \left( \frac{\alpha_{i j}}{\prod_{l \neq i, j} \chi_{i l} \chi_{l j} } \right)_{i j} = \iota_* \left( e_{z_{ij}, Y}(\alpha) \right)_{i j} \] 
    By \cite[Corollary 2.2]{teq}, $\iota_*$ is an isomorphism after tensoring with $Q$. Since $A_T^\bullet(Y^T)$ is free as an $R$-module, the desired equality follows.

    For (2), since $R \subseteq Q$, it is enough to compute after localising. By (1), 
    \[ \alpha \beta = \sum_{0 \leq i \neq j \leq n} \frac{\alpha_{i j} \beta_{i j} }{\prod_{l \neq i, j} \chi_{i l} \chi_{l j}} \iota_{i j *} (1) \] 
    Now, $p_* \circ \iota_{ij *}$ is the identity on $Q$ since $p \circ \iota_{i j}$ is a map of a point to itself. Thus, by linearity we have 
    \[ \langle \alpha, \beta \rangle_T = p_*(\alpha \beta) =  \sum_{0 \leq i \neq j \leq n} \frac{\alpha_{i j} \beta_{i j} }{\prod_{l \neq i, j} \chi_{i l} \chi_{l j}} \]
\end{proof}

\begin{proof}[Proof of Proposition 3.3]
  It is enough to show that for $0 \leq i, j \leq n$, $\langle [E_i], [E_j] \rangle_T = \delta_{i j} (-1)^{n-1}$. By \Cref{chars}, 
  $e_{z_{i j}, Y}([E_i]) = (\prod_{l \neq i, j} \chi_{l j} )^{-1}$. Hence by \Cref{eq_ids}, $\langle [E_i], [E_j] \rangle_T = 0$ when $i \neq j$ (since 
  $E_i$ and $E_j$ share no fixed points) and 
  \[ \langle [E_i], [E_i] \rangle_T = \sum_{s \neq i} \frac{\prod_{l \neq i, s} \chi_{i l} }{ \prod_{l \neq i, s} \chi_{l s}} \] This is seen to be $(-1)^{n-1}$ 
  by the following observation in \cite[Lemma 4.2]{twistedmilnor}: treating $R$ as a polynomial ring in $t_i$ over $\Z[t_0, \dots, \hat{t}_i, \dots, t_n]$, by 
  Lagrange interpolation it is enough to show that the polynomial \[f(t_i) = \sum_{s \neq i} \frac{\prod_{l \neq i, s} \chi_{i l} }{ \prod_{l \neq i, s} \chi_{l s}}\] 
  of degree at most $n-1$ evaluated at $t_j$ for each $j \neq i$ is $(-1)^{n-1}$. Clearly, $\prod_{l \neq i, s} \chi_{i l}$ evaluated at $t_j$ is 0 if $j \neq s$, thus 
  \[ f(t_j) =  \frac{\prod_{l \neq i, j} \chi_{j l}}{\prod_{l \neq i, j} \chi_{l j} } = (-1)^{n-1}  \] for $j \neq i$.
\end{proof}

\begin{prop} \label{universal_inj}
  If all of the eigenvalues of $\varphi$ are in $k$, the map $(j \times \id_S)_* \circ (\pi|_Z \times \id_S)^*$ is injective for any smooth projective variety $S$ over $k$.
\end{prop}

\begin{proof}
  We have that $ A^\bullet(Z \times_k S) = \bigoplus_{0 \leq i \leq n} A^\bullet(S)$, with the images of the classes of the irreducible components of $Z \times_k S$ under 
  $(j \times \id_S)_* \circ (\pi|_Z \times \id_S)^*$ being the classes $E|_i \times_k S$ in $A^\bullet(Y \times_k S)$, 
  i.e. $\gamma_i \times 1_S$ ($1_S = [S]$), $0 \leq i \leq n$. 
  Note that the homomorphism is $A^\bullet(S)$-linear, so it suffices to show that the $\gamma_i$ are $A^\bullet(S)$-linearly independent 
  in $A^\bullet(Y \times_k S)$. We define a relative Poincaré pairing $A^\bullet(Y \times_k S) \times A^\bullet(Y \times_k S) \to A^\bullet(S) $ by 
  $\langle \alpha, \beta \rangle_S = (p \times \id_S)_* (\alpha \cdot \beta)$, where $p: Y \to \Spec k$ is the structure morphism. Note that this is 
  $A^\bullet(S)$ bilinear. Then for $0 \leq i, j \leq n$, $\langle \gamma_i \times 1_S, \gamma_j \times 1_S \rangle_S = \deg (\gamma_i \cdot \gamma_j) 
  \cdot 1_S = \delta_{i j} (-1_S)^{n-1}$ by \Cref{ortho}. Linear independence follows. 
\end{proof}

\begin{proof}[Proof of Theorem 1.1]

  Applying base change by $K/k$, we obtain the commutative diagram of Cartesian squares: 
\[\begin{tikzcd}[cramped]
	{(Z \times_k K) \times_K (S \times_k K)} && {(E|_Z \times_k K) \times_K (S \times_k K)} && {(Y \times_k K) \times_K (S \times_k K)} \\
	\\
	{Z \times_k S} && {E|_Z \times_k S} && {Y \times_k S}
	\arrow[from=1-1, to=3-1]
	\arrow[from=1-3, to=1-1]
	\arrow[from=1-3, to=1-5]
	\arrow[from=1-3, to=3-3]
	\arrow[from=1-5, to=3-5]
	\arrow["{\pi|_Z \times \id_S}", from=3-3, to=3-1]
	\arrow["{j \times \id_S}", from=3-3, to=3-5]
\end{tikzcd}\] which induces the commutative diagram on Chow rings 
\[\begin{tikzcd}[cramped]
	{A^\bullet(Z_K \times_K S_K)} && {A^\bullet((E|_Z)_K \times_K S_K)} && {A^\bullet(Y_K \times_K S_K)} \\
	\\
	{A^\bullet(Z \times_k S)} && {A^\bullet(E|_Z \times_k S)} && {A^\bullet(Y \times_k S)}
	\arrow[from=1-1, to=1-3]
	\arrow[from=1-3, to=1-5]
	\arrow[from=3-1, to=1-1]
	\arrow["{(\pi|_Z \times \id_S)^*}", from=3-1, to=3-3]
	\arrow[from=3-3, to=1-3]
	\arrow["{(j \times \id_S)_*}", from=3-3, to=3-5]
	\arrow[from=3-5, to=1-5]
\end{tikzcd}\]
where we write $S_K$ for $S \times_k K$. The lefthand vertical map is injective. Indeed, by \Cref{compute_Z}, $Z \cong \Spec L$. Since 
$K \otimes_k L \cong K \otimes_k (\prod_{1 \leq i \leq m} K) \cong \prod_{0 \leq i \leq n} K  $, we need only check that the obvious map 
$\bigoplus_{1 \leq i \leq m}  A^\bullet(S_K) \to \bigoplus_{0 \leq i \leq n} A^\bullet(S_K)$ is injective, which is clear.  
Thus, injectivity of the composite of the top row implies injectivity of the composite of the bottom 
row, i.e. $(j \times \id_S)_* \circ (\pi|_Z \times \id_S)^*$. The injectivity of the top row is exactly \Cref{universal_inj}, in the case of $V \otimes_k K$ and 
$\varphi \otimes 1 \in \End(V \otimes_k K)$, since by definition the eigenvalues of $\varphi \otimes 1$ are in $K$. The theorem then follows from \Cref{crit}.

\end{proof}

\begin{cor} \label{invariant}
  Suppose $\varphi, \varphi' \in \End(V)$ satisfy the hypotheses of \Cref{main_thm}. If the associated varieties $Y$ and $Y'$ are isomorphic, then 
  $L = k[\varphi]$ and $L' = k[\varphi']$ are isomorphic $k$-algebras.
\end{cor}

\begin{proof}
  By hypothesis, $L = \prod_{i} K$ and $L' = \prod_{j} K'$, $K'/k$ Galois. 
  Since $\dim_k L = \dim_k L'$ it is the same to show $Y \cong Y' \iff K \cong K'$. 
  First, assume $Y \cong Y'$. Then $\Hom(\M(\Spec K')(n-1), \M(Y)) \cong \Hom(\M(\Spec K'), \M(Y))$ as abelian groups. The motive $\M(\proj^n_k)$ decomposes as 
  $\bigoplus_{0 \leq i \leq n} \Z(i) $, so $\Hom(\M(\Spec K)(n-1), \M(\proj^n_k)(m)) = \Hom(\M(\Spec K)(n-1), \Z(n-1)) = \Z $ for all $m \geq 0$. Hence by \Cref{main_thm},
  $\Z^{n-1} \oplus \Hom(\M(\Spec K'), \M(\Spec L)) \cong \Z^{n-1} \oplus \Hom(\M(\Spec K'), \M(\Spec L'))$. By definition, $\Hom(\M(\Spec K'), \M(\Spec L'))$ is 
  $A^0(\Spec K' \times_k \Spec L') = A^0(\coprod_{0 \leq i \leq n} \Spec K' ) = \Z^{n+1}$. Thus, $\Hom(\M(\Spec K'), \M(\Spec L)) \cong \Z^{n+1}$. But this 
  means $\Spec L \times_k K'$ has $n+1$ irreducible components, but since $\dim_k L = n+1$, this must mean $\Spec L$ has a $K'$-point, i.e. there is an embedding $K \mono K'$.
  By reversing the roles of $K$ and $K'$, we see there is also an embedding $K' \mono K$, whence $K \cong K'$.
\end{proof}

\printbibliography

\end{document}